 \documentclass[10.5pt,twoside]{article}
  \usepackage{geometry}
\usepackage{amssymb}
\usepackage{amsmath}
\usepackage{amsthm}
\usepackage{mathrsfs}
\usepackage{cite}
\usepackage{color}
\usepackage{changes}
\usepackage[colorlinks,
linkcolor=blue,
anchorcolor=blue,
citecolor=red
]{hyperref}

 \def\rr{\mathbb{R}}
\def\zz{\mathbb{Z}}
\def\nn{\mathbb{N}}
\def\cp{\mathcal{P}}
\def\cf{\mathcal{F}}

\def\cs{\mathcal{S}}

\def\supp{{\rm{\ supp\ }}}

\def\rn{\mathbb{R}^n}
\def\zn{\mathbb{Z}^n}

\def\no{\nonumber}

\def\les{\lesssim}

\newtheorem{thm}{Theorem}[section]
\newtheorem{rem}[thm]{Remark}
\newtheorem{lem}[thm]{Lemma}

\newtheorem{cor}[thm]{Corollary}
\newtheorem{defn}[thm]{Definition}

\def\XXint#1#2#3{{
\setbox0=\hbox{$#1{#2#3}{\int}$}
\vcenter{\hbox{$#2#3$}}\kern-.5\wd0}}

\allowdisplaybreaks[4]
\numberwithin{equation}{section}
\begin{document}
\title{The Weighted Grand Herz-Morrey-Lizorkin-Triebel Spaces with Variable Exponents    \footnotetext{$\ast$ The corresponding author J. S. Xu  jingshixu@126.com}}
\author{Shengrong Wang\textsuperscript{1}, Pengfei Guo\textsuperscript{1}, Jingshi Xu\textsuperscript{2,3,4*}\\
{\scriptsize  \textsuperscript{1} School of Mathematics and Statistics, Hainan Normal University, Haikou, 571158, China }\\
{\scriptsize  \textsuperscript{2}School of Mathematics and Computing Science, Guilin University of Electronic Technology, Guilin 541004, China} \\
{\scriptsize  \textsuperscript{3} Center for Applied Mathematics of Guangxi (GUET), Guilin 541004, China}\\
{\scriptsize  \textsuperscript{4}Guangxi Colleges and Universities Key Laboratory of Data Analysis and Computation, Guilin 541004, China}}
\date{}

\maketitle

{\bf Abstract.} Let a vector-valued sublinear operator satisfy the size condition and be bounded on weighted Lebesgue spaces with variable exponent. Then we obtain its boundedness on weighted grand Herz-Morrey spaces with variable exponents. Next we introduce weighted grand Herz-Morrey-Triebel-Lizorkin spaces with variable exponents and provide their equivalent quasi-norms via maximal functions.

 {\bf Key words and phrases.} subilinear operator, vector-valued inequality, Muckenhoupt weight, variable exponent, grand Herz-Morrey space,  Triebel-Lizorkin space
 
{\bf Mathematics Subject Classification (2020).} 42B25, 46E30, 46E35

\section{Introduction}
The theory of variable exponent function spaces has been rapidly developed after Kov\'a\v{c}ik and R\'akosn\'{\i}k \cite{kr-1}  gave  basic properties of Lebesgue spaces with  variable exponent.
In particular, the boundedness of the Hardy-littlewood maximal operator on the variable exponent Lebesgue spaces, was studied in \cite{cfj-2,dl-3,ne1,dl-4,dhhm1}.
Cruz-Uribe, Fiorenza and Neugebauer \cite{cruz-1} extended the classic Muckenhoupt $A_p$ weight in \cite{muc} to variable exponent and proved the equivalence between the $A_{p (\cdot)}$ weight condition and the boundedness of the Hardy Littlewood maximal operator on the correspondingly weighted Lebesgue space with variable exponent.

The boundedness of some sublinear operators, which include the Hardy-Littlewood maximal operator, was considered by many authors recently. Indeed, 
the boundedness of some sublinear operators on weighted variable Herz-Morrey spaces was obtained by Wang and Shu in \cite{wang}. Then the boundedness of vector-valued sublinear operators on weighted Herz-Morrey spaces with variable exponents was proved in \cite{wx3}. Next, 
the grand variable Herz spaces were introduced, and the boundedness of sublinear operators on them was obtained by 
Nafis, Rafeiro and Zaighum in \cite{nrz1}. Finally, The weighted grand Herz-Morrey type spaces were introduced, and the boundedness of sublinear operators and their multilinear commutators on them was obtained by Zhang, He and Zhang in \cite{zhz1}.

Mixing together Herz spaces, Besov spaces and Triebel-Lizorkin spaces, Yang and the third author of the paper \cite{xy-5,xu-6} introduced the Herz-type Besov spaces and Triebel-Lizorkin spaces.
Then, Shi and the third author of the paper \cite{sx-1} extended them to variable exponent  and obtained their equivalent quasi-norms.

Motivated by the mentioned works, in this paper, we consider the boundedness of vector-valued sublinear operators on weighted grand Herz-Morrey spaces with variable exponents, and introduce
weighted grand Herz-Morrey-Triebel-Lizorkin spaces with variable exponents.
The plan of the paper is as follows. 
In Section \ref{wbl-s-1}, we collect some notations. 
In Section \ref{wbl-s-2}, we give the boundedness of vector-valued sublinear operators on weighted grand Herz-Morrey spaces with variable exponents.
In Section \ref{wbl-s-3}, we introduce weighted grand Herz-Morrey-Triebel-Lizorkin spaces with variable exponent, and obtain their equivalent quasi-norms  by Peetre’s maximal operators.

Now, we make some conventions on notation. Let $\nn$ be the collection of all natural numbers and $\nn_0 = \nn \cup \{0\}$. Let $\zz$ be the collection of all integers.
Let $\rn$ be the $n$-dimensional Euclidean space, where $n \in \nn$. In the sequel, $C$ denotes positive constants, but it may change from line to line. For any quantities $A$ and $B$, if there exists a constant $C>0$ such that $A\leq CB$, we write $A \lesssim B$. If $A\lesssim B$ and $B \lesssim A$, we write $A \sim B$.
For each $k \in \mathbb{Z}$, we define
${B_k}: = \{ x \in \rn:| x | \leq {2^k} \},$ $D_k: = B_k\backslash B_{k - 1}=\{ x\in \rn: 2^{k-1}<|x| \leq 2^k\},$ $\chi _k: = \chi _{D_k},$ $\widetilde{\chi}_m=\chi_{m},$
$m \geq 1,$ $\widetilde{\chi}_0=\chi_{B_0}.$
 For a measurable function $q(\cdot)$ on $\rn$ taking values in $\ (0,\infty)$, we denote ${q^- }: =  {\rm ess}\inf_{x \in \rn} p(x),$ ${q^+ }: =  {\rm ess} \sup_{x \in\rn} q(x).$
The set $\cp(\rn)$ consists of all $q(\cdot)$ satisfying $q^->1$ and $q^+<\infty;$
$\cp_{0}(\rn)$ consists of all $q(\cdot)$ satisfying $q^->0$ and $q^+<\infty$.
\section{Notations and preliminaries}\label{wbl-s-1} 
In this section, we first recall some definitions and notations. 
Let $q(\cdot) \in \cp_0(\rn)$.
Then the Lebesgue space with variable exponent $L^{q(\cdot)}(\rn)$ is defined
to be the set of all measurable functions $f$ such that
\[\| f \|_{L^{q(\cdot)}(\rn)}: = \inf \bigg\{ \lambda  > 0:\int_{\rn} \bigg( \frac{|f(x)|}{\lambda } \bigg)^{q(x)}{\rm d} x \leq 1  \bigg\}<\infty.\]

 Let $p(\cdot) \in \mathcal{P}(\rn)$ and $w$ be a weight which is a nonnegative measurable function on $\rn$.
Then the weighted variable exponent Lebesgue space $L^{q(\cdot)}(w)$ is  the set of all complex-valued measurable function $f$
such that $f w \in L^{q(\cdot)}(\rn)$. The space $L^{q(\cdot)}(w)$ is a Banach space equipped with the norm
\[\|f\|_{L^{q(\cdot)}(w)}:=\|f w\|_{L^{q(\cdot)}(\rn)}.\]

Let $q(\cdot) \in \cp_0(\rn)$. The set $L^{q(\cdot)} _{\rm loc}(w)$ consists of all measurable functions $f$ such that $f\chi_{K} \in L^{q(\cdot)}(w)(\rn)$ for all compact subsets $K \subset \rn$.
Denote by $L^\infty(\rn)$ the set of all measurable functions $f$ such that
$ \|f\|_{L^\infty} := \mathop{{\rm ess}\sup}_{y \in \rn} |f(y)| < \infty.$

Let $f\in L^{1}_{\rm loc}(\rn)$. Then the standard Hardy-Littlewood maximal function of $f$ is defined by
\[\mathcal{M}f(x) := \sup_{B \ni x} \frac{1}{| B |}\int_B |f(y)|{\rm d}y , \ \forall x \in \mathbb{R}^n,\]
where the supremum is taken over all balls containing $x$ in $\rn$. 
We also define 
$\mathcal{M}_t(f) :=(\mathcal{M}(|f|^t))^{1/t},$ $t \in (0,\infty).$

\begin{defn}\label{cd-1}
Let $\alpha(\cdot)$ be a real-valued measurable function on $\rn$.

{\rm (i)} The function $\alpha(\cdot)$ is locally $\log$-H\"{o}lder continuous if there exists a constant $C_{\log}(\alpha)$ such that
\[|\alpha (x) - \alpha (y)| \leq \frac{C_{\log}(\alpha)}{\log ( e + 1/|x - y|)}, \ x,y \in \mathbb{R}^n,\ |x - y| < \frac{1}{2}.\]
Denote by $C_{\rm loc}^{\log}(\rn)$ the set of all locally log-H\"{o}lder continuous

{\rm (ii)} The function $\alpha(\cdot)$ is $\log$-H\"{o}lder continuous at the origin if there exists a constant $C_0$ such that
\[|\alpha (x) - \alpha (0)| \leq \frac{C_0} {\log ( e + 1/| x | )}, \ \forall x \in \rn.\]
Denote by $C^{\log}_0(\rn)$ the set of all $\log$-H\"{o}lder continuous functions at the origin.

{\rm (iii)} The function $\alpha(\cdot)$ is $\log$-H\"{o}lder continuous at infinity if there exists $\alpha_{\infty}\in \mathbb{R}$ and a constant $C_\infty$ such that
\[|\alpha (x) - \alpha _\infty|\leq \frac{C_\infty} {\log ( e + | x | )}, \ \forall x \in \rn.\]
Denote by $C^{\log}_\infty(\rn)$ the set of all $\log$-H\"{o}lder continuous functions at infinity.

{\rm (iv)} The function $\alpha(\cdot)$ is global $\log$-H\"{o}lder continuous if $\alpha(\cdot)$ are both locally $\log$-H\"{o}lder continuous and $\log$-H\"{o}lder continuous at infinity.
Denote by $C^{\log}(\rn)$ the set of all global $\log$-H\"{o}lder continuous functions.
\end{defn}

\begin{defn}\label{dg1}
Let $q(\cdot) \in \mathcal{P}(\rn)$, a nonnegative measurable function $w$ is said to be in $A_{q(\cdot)}$,
if 
\[ \|w\|_{A_{q(\cdot)}}:= \sup_{\text{all ball } B \subset \rn} \frac{1}{|B|}\|w\chi_B\|_{L^{q(\cdot)}(\rn)} \|w^{-1}\chi_B\|_{L^{q^{\prime}(\cdot)}(\rn)} < \infty.\]
\end{defn}

\begin{defn}\label{dg2}
Let $q(\cdot) \in \mathcal{P}(\mathbb{R}^n)$, a nonnegative measurable function $w$ is said to be in $\tilde{A}_{q(\cdot)}$,
if 
\[\|w\|_{\tilde{A}_{q(\cdot)}} := \sup_{\text{all ball } B \subset \rn} \frac{1}{|B|}\|w^{1/q(\cdot)}\chi_B\|_{L^{q(\cdot)}(\rn)} \|w^{-1/q(\cdot)}\chi_B\|_{L^{q^{\prime}(\cdot)}(\rn)}  < \infty.\]
\end{defn}

\begin{defn}  Let $w$ be weights on $\rn$, $q(\cdot)\in\cp(\rn),$ $p\in(1,\infty)$, $\lambda \in [0, \infty)$, $\alpha(\cdot) \in L^\infty(\rn)$ and $\theta>0$.
  The homogeneous  weighted grand Herz-Morrey space $M\dot{K}_{q(\cdot),\lambda}^{\alpha (\cdot),p),\theta}(w)$ is defined respectively by
\[M\dot{K}_{q(\cdot),\lambda}^{\alpha (\cdot),p),\theta}(w):= \Big\{f\in L^{q(\cdot)}_{\rm loc}(\rn \setminus \{0\},w):\|f\|_{M\dot{K}_{q(\cdot),\lambda}^{\alpha (\cdot),p),\theta}(w)}<\infty\Big\},\]
where
\[\|f\|_{M\dot{K}^{\alpha(\cdot),p),\theta}_{q(\cdot),\lambda}(w)}:= \sup_{\delta>0} \sup_{k_0\in \zz}2^{-k_0\lambda} \bigg( \delta^\theta \sum_{k=-\infty}^{k_0} \big\| 2^{k\alpha(\cdot)} f\chi_k \big\|_{L^{q(\cdot)}(w)}^{p(1+\delta)} \bigg)^{\frac{1}{p(1+\delta)}}.\]
\end{defn}

\begin{lem}[see {\cite[Lemma 4.1]{zhz1}}]\label{ghm-L1}
Let $w$ be a weight on $\rn$, $q(\cdot)\in\cp(\rn),$ $p\in(1,\infty)$, $\lambda \in [0, \infty)$, $\alpha(\cdot) \in L^\infty(\rn)$ and $\theta>0$. If $\alpha(\cdot) \in C^{\log}_0(\rn)\cap C^{\log}_{\infty}(\rn),$
then for all $f\in L^{q(\cdot)}_{\rm loc}(\rn\backslash\{0\},w),$
\begin{align*}
\|f\|_{M\dot{K}^{\alpha(\cdot),p),\theta}_{q(\cdot),\lambda}(w)} &
\sim \max \Bigg\{ \sup_{\delta>0} \sup_{k_0\leq 0,k_0\in \zz}2^{-k_0\lambda} \bigg( \delta^\theta \sum_{k=-\infty}^{k_0} 2^{k\alpha(0)p(1+\delta)} \big\| f\chi_k \big\|_{L^{q(\cdot)}(w)}^{p(1+\delta)} \bigg)^{\frac{1}{p(1+\delta)}}, \\
&   \sup_{\delta>0} \sup_{k_0 > 0,k_0\in \zz}2^{-k_0\lambda} \bigg( \delta^\theta \sum_{k=-\infty}^{-1} 2^{k\alpha(0)p(1+\delta)} \big\| f\chi_k \big\|_{L^{q(\cdot)}(w)}^{p(1+\delta)} \\
& \quad + \delta^\theta \sum_{k=0}^{k_0} 2^{k\alpha_\infty p(1+\delta)} \big\| f\chi_k \big\|_{L^{q(\cdot)}(w)}^{p(1+\delta)} \bigg)^{\frac{1}{p(1+\delta)}} \Bigg\}.
\end{align*}
\end{lem}

The following lemma  has been proved by Noi and Izuki in \cite{in-1}.
\begin{lem}\label{tt-L5}
If $q(\cdot)\in C^{\log}(\rn)\cap \mathcal{P}(\rn)$ and $w \in A_{q(\cdot)}$,
then there exist constants $\delta_{1}$, $\delta_{2}\in (0,1)$ and $C>0$ such that for all balls $B$ in $\rn$ and all measurable subsets $S\subset B,$
\begin{equation*}
\frac{\|\chi _S\|_{L^{q(\cdot)}(w)}}{\| \chi _B \|_{L^{q(\cdot)}(w)}} \leq C\bigg( \frac{| S |}{| B |} \bigg)^{\delta _1}\ \text{and} \ \frac{\|\chi _S\|_{L^{q^{\prime}( \cdot )}(w^{-1})}}{\|\chi _B \|_{L^{q^{\prime}( \cdot )}(w^{-1})}} \leq C\bigg( \frac{| S |}{| B |} \bigg)^{\delta _2}.
\end{equation*}
\end{lem}
\begin{rem}
 If $w^{q(\cdot)} \in \tilde{A}_{q(\cdot)}$, then $w^{-q^{\prime}(\cdot)} \in \tilde{A}_{p^{\prime}(\cdot)}$  by Definitions \ref{dg2}. 
 Then if  $w^{q(\cdot)} \in \tilde{A}_{q(\cdot)}$, we have $w \in A_{q(\cdot)}$ and $w^{-1} \in A_{q^{\prime}(\cdot)}$ by Definitions \ref{dg1}. 
 \end{rem}
\section{Boundedness of the vector-valued sublinear operator on weighted grand Herz-Morrey spaces }\label{wbl-s-2}
In this section, we give the boundedness of vector-valued sublinear operators on weighted grand Herz-Morrey spaces with variable exponents.

\begin{thm} \label{vT1}
Let $r, p \in (1,\infty)$,
$q(\cdot)\in C^{\log}(\rn)\cap \mathcal{P}(\rn)$, $\theta>0$,
 $\alpha(\cdot) \in L^{\infty}(\mathbb{R}^{n})\cap C^{\log}_0(\rn)\cap C^{\log}_{\infty}(\rn)$, $w \in A_{q(\cdot)}$,
 $- n\delta_1 <\alpha(0),\ \alpha_\infty <n\delta_2$,  where $\delta_1,\ \delta_2 \in (0,1)$ are the constants in Lemma \ref{tt-L5} for $q(\cdot)$.
Suppose that $T$ is a sublinear operator satisfies the size condition,
\begin{align}\label{vf-t2}
|Tf(x)| \leq C \int_{\rn} |x-y|^{-n} |f(y)|{\rm d}y
\end{align}
for all $f \in L^1_{\rm loc}(\rn)$ and a.e. $x \notin \supp f$. If the sublinear operator $T$ satisfies vector-valued inequality on $L^{q(\cdot)}(w)$,
\begin{equation}\label{vf-t1}
\bigg\| \bigg(\sum^\infty_{j=1}|Tf_j|^r \bigg)^{\frac{1}{r}}\bigg\|_{L^{q(\cdot)}(w)} \leq C \bigg\| \bigg(\sum^\infty_{j=1}|f_j|^{r} \bigg)^{\frac{1}{r}}\bigg\|_{L^{q(\cdot)}(w)}
\end{equation}
for all sequences$\{f_j\}^\infty_{ j=1}$ of locally integrable functions on $\rn$, then
\begin{equation}\label{vf-t3}
\bigg\| \bigg( \sum^\infty_{j=1} |Tf_j|^r \bigg)^{\frac{1}{r}}\bigg\|_{M\dot{K}_{q(\cdot),\lambda}^{\alpha (\cdot),p),\theta}(w)} \leq C  \bigg\| \bigg( \sum^\infty_{j=1} |f_j|^r \bigg)^{\frac{1}{r}}\bigg\|_{M\dot{K}_{q(\cdot),\lambda}^{\alpha (\cdot),p),\theta}(w)}.
\end{equation}
where the constant $C>0$ is independent of $\{f_j\}^\infty_{ j=1}$.
\end{thm}

\begin{lem}[see {\cite[Corollary 3.2]{cw-1}}]\label{vL13}
Let $q(\cdot) \in \cp(\rn)$ and $w$ be a weight. If the maximal operator $\mathcal{M}$ is bounded on $L^{q(\cdot)}(w)$ and $L^{q'(\cdot)}(w^{-1})$ and $r\in (1,\infty)$, then there is a positive constant $C$ such that
\[ \bigg\| \bigg(\sum^\infty_{j=1} (\mathcal{M}f_j)^r\bigg)^{\frac{1}{r}}\bigg\|_{L^{q(\cdot)}(w)} \leq C \bigg\| \bigg(\sum^\infty_{j=1} |f_j|^r\bigg)^{\frac{1}{r}}\bigg\|_{L^{q(\cdot)}(w)}. \]
\end{lem}

From Theorem \ref{vT1} and Lemma \ref{vL13}, we obtain the following corollary.
\begin{cor}\label{vC1}
Let $r, p \in (1,\infty)$,
$q(\cdot)\in C^{\log}(\rn)\cap \mathcal{P}(\rn)$, $\theta>0$,
 $\alpha(\cdot) \in L^{\infty}(\mathbb{R}^{n})\cap C^{\log}_0(\rn)\cap C^{\log}_{\infty}(\rn)$, $w \in A_{q(\cdot)}$,
 $ - n\delta_1 <\alpha(0),\ \alpha_\infty <n\delta_2$, where $\delta_1,\ \delta_2 \in (0,1)$ are the constants in Lemma \ref{tt-L5} for $q(\cdot)$. Then
\[\bigg\| \bigg( \sum^\infty_{j=1} |\mathcal{M}f_j|^r \bigg)^{\frac{1}{r}}\bigg\|_{M\dot{K}_{q(\cdot),\lambda}^{\alpha (\cdot),p),\theta}(w)} \leq C  \bigg\| \bigg( \sum^\infty_{j=1} |f_j|^r \bigg)^{\frac{1}{r}}\bigg\|_{M\dot{K}_{q(\cdot),\lambda}^{\alpha (\cdot),p),\theta}(w)},\]
where the constant $C>0$ is independent of $\{f_j\}^\infty_{ j=1}$ of locally integrable functions on $\rn$.
\end{cor}

\begin{proof}[\bf Proof of Theorem \ref{vT1}]
 Since bounded functions with compact support are dense in $M\dot{K}_{q_2(\cdot),\lambda}^{\alpha (\cdot),p),\theta}(w)$, we only consider $f$ is a bounded function with compact support and write
\[f_j(x)=\sum^\infty_{l=-\infty}f^l_j \chi_{l}=:\sum^\infty_{l=-\infty}f^l_j, \quad j\in \nn.\]
By Lemma \ref{ghm-L1}, we have
\begin{align*}
& \bigg\| \bigg( \sum^\infty_{j=1} |Tf_j|^r \bigg)^{\frac{1}{r}}\bigg\|_{M\dot{K}_{q(\cdot),\lambda}^{\alpha (\cdot),p),\theta}(w)} \\
 & \quad \sim  \max \Bigg\{ \sup_{\delta>0} \sup_{k_0\leq 0,k_0\in \zz}2^{-k_0\lambda} \bigg( \delta^\theta \sum_{k=-\infty}^{k_0} 2^{k\alpha(0)p(1+\delta)} \bigg\| \bigg( \sum_{j=1}^\infty \bigg| \sum_{l=-\infty}^\infty Tf^l_j \bigg|^r \bigg)^{\frac{1}{r}} \chi_k \bigg\|_{L^{q(\cdot)}(w)}^{p(1+\delta)} \bigg)^{\frac{1}{p(1+\delta)}}, \\
& \quad   \sup_{\delta>0} \sup_{k_0 > 0,k_0\in \zz}2^{-k_0\lambda} \bigg( \delta^\theta \sum_{k=-\infty}^{-1} 2^{k\alpha(0)p(1+\delta)} \bigg\| \bigg( \sum_{j=1}^\infty \bigg| \sum_{l=-\infty}^\infty Tf^l_j \bigg|^r \bigg)^{\frac{1}{r}} \chi_k \bigg\|_{L^{q(\cdot)}(w)}^{p(1+\delta)} \\
& \quad \quad + \delta^\theta \sum_{k=0}^{k_0} 2^{k\alpha_\infty p(1+\delta)} \bigg\| \bigg( \sum_{j=1}^\infty \bigg| \sum_{l=-\infty}^\infty Tf^l_j \bigg|^r \bigg)^{\frac{1}{r}} \chi_k \bigg\|_{L^{q(\cdot)}(w)}^{p(1+\delta)} \bigg)^{\frac{1}{p(1+\delta)}} \Bigg\} \\
 & \quad=:\max \{E,F+G\}.
\end{align*}

Since the estimation of $F$ is essentially similar to that of $E$, so we suffice to obtain $E$ and $G$ are bounded on weighted grand Herz-Morrey spaces  with variable exponents. It is easy to see that
\[E \les \sum^3_{i=i} E_i \text{ and } G \les \sum^3_{i=i}G_i,\]
where
\[E_1 := \sup_{\delta>0} \sup_{k_0\leq 0,k_0\in \zz}2^{-k_0\lambda} \bigg( \delta^\theta \sum_{k=-\infty}^{k_0} 2^{k\alpha(0)p(1+\delta)} \bigg\| \bigg( \sum_{j=1}^\infty \bigg| \sum_{l=-\infty}^{k-2} Tf^l_j \bigg|^r \bigg)^{\frac{1}{r}} \chi_k \bigg\|_{L^{q(\cdot)}(w)}^{p(1+\delta)} \bigg)^{\frac{1}{p(1+\delta)}},\]
\[E_2 := \sup_{\delta>0} \sup_{k_0\leq 0,k_0\in \zz}2^{-k_0\lambda} \bigg( \delta^\theta \sum_{k=-\infty}^{k_0} 2^{k\alpha(0)p(1+\delta)} \bigg\| \bigg( \sum_{j=1}^\infty \bigg| \sum_{l=k-1}^{k+1} Tf^l_j \bigg|^r \bigg)^{\frac{1}{r}} \chi_k \bigg\|_{L^{q(\cdot)}(w)}^{p(1+\delta)} \bigg)^{\frac{1}{p(1+\delta)}},\]
\[E_3 := \sup_{\delta>0} \sup_{k_0\leq 0,k_0\in \zz}2^{-k_0\lambda} \bigg( \delta^\theta \sum_{k=-\infty}^{k_0} 2^{k\alpha(0)p(1+\delta)} \bigg\| \bigg( \sum_{j=1}^\infty \bigg| \sum_{l=k+2}^\infty Tf^l_j \bigg|^r \bigg)^{\frac{1}{r}} \chi_k \bigg\|_{L^{q(\cdot)}(w)}^{p(1+\delta)} \bigg)^{\frac{1}{p(1+\delta)}},\]
\[G_1:= \sup_{\delta>0} \sup_{k_0 > 0,k_0\in \zz}2^{-k_0\lambda} \bigg( \delta^\theta \sum_{k=0}^{k_0} 2^{k\alpha_\infty p(1+\delta)} \bigg\| \bigg( \sum_{j=1}^\infty \bigg| \sum_{l=-\infty}^{k-2} Tf^l_j \bigg|^r \bigg)^{\frac{1}{r}} \chi_k \bigg\|_{L^{q(\cdot)}(w)}^{p(1+\delta)} \bigg)^{\frac{1}{p(1+\delta)}}.\]
\[G_2:= \sup_{\delta>0} \sup_{k_0 > 0,k_0\in \zz}2^{-k_0\lambda} \bigg( \delta^\theta \sum_{k=0}^{k_0} 2^{k\alpha_\infty p(1+\delta)} \bigg\| \bigg( \sum_{j=1}^\infty \bigg| \sum_{l=k-1}^{k+1} Tf^l_j \bigg|^r \bigg)^{\frac{1}{r}} \chi_k \bigg\|_{L^{q(\cdot)}(w)}^{p(1+\delta)} \bigg)^{\frac{1}{p(1+\delta)}}.\]
\[G_3:= \sup_{\delta>0} \sup_{k_0 > 0,k_0\in \zz}2^{-k_0\lambda} \bigg( \delta^\theta \sum_{k=0}^{k_0} 2^{k\alpha_\infty p(1+\delta)} \bigg\| \bigg( \sum_{j=1}^\infty \bigg| \sum_{l=k+2}^\infty Tf^l_j \bigg|^r \bigg)^{\frac{1}{r}} \chi_k \bigg\|_{L^{q(\cdot)}(w)}^{p(1+\delta)} \bigg)^{\frac{1}{p(1+\delta)}}.\]
If $l \leq k-2,$ we deduce that
$|x-y| \geq |x|-|y|>2^{k-1}-2^l \geq 2^{k-2} ,\, x\in D_k,\, y \in D_{l}.$
Then by (\ref{vf-t2}) for $\forall x \in D_k$, we have
\[|Tf^l_j(x)| \lesssim 2^{-kn} \int_{\rn} |f^l_j(y)| {\rm d}y.\]
Thus, by Minkowski's inequality, H\"{o}lder's inequality, Definition \ref{dg1} and Lemma \ref{tt-L5}, we obtain
\begin{align}\label{vsgh-1}
&\bigg\|\bigg(\sum^\infty_{j=1}\bigg|\sum^{k-2}_{l=-\infty}  Tf^l_j\bigg|^r\bigg)^{\frac{1}{r}} \chi_k \bigg\|_{L^{q(\cdot)}(w)} \no\\
& \quad \lesssim \bigg\| \bigg(\sum^\infty_{j=1} \bigg(\sum^{k-2}_{l=-\infty} 2^{-kn}\int_{\rn} |f^l_j(y)|{\rm d}y\bigg)^{r}\bigg)^{\frac{1}{r}}  \chi_{k}\bigg\|_{L^{q(\cdot)}(w)} \nonumber\\
& \quad  \lesssim \bigg\| \sum^{k-2}_{l=-\infty} 2^{-kn}\int_{\rn} \bigg(\sum^\infty_{j=1}|f^l_j(y)|^{r}\bigg)^{\frac{1}{r}}{\rm d}y \chi_{k}\bigg\|_{L^{q(\cdot)}(w)} \no\\
& \quad \les \sum^{k-2}_{l=-\infty} 2^{-kn} \|\chi_{B_k}\|_{L^{q(\cdot)}(w)} \bigg\|\bigg(\sum^\infty_{j=1} |f^l_j|^{r}\bigg)^{\frac{1}{r}} w \chi_{l}\bigg\|_{L^{q(\cdot)}(\rn)} \|\chi_{l}w^{-1}\|_{L^{q'(\cdot)}(\rn)} \nonumber\\
& \quad \les \sum^{k-2}_{l=-\infty} 2^{-kn} |B_k| \|\chi_{B_k}\|^{-1}_{L^{q'(\cdot)}(w^{-1})}
\|\chi_{B_{l}}\|_{L^{q'(\cdot)}(w^{-1})} \bigg\|\bigg(\sum^\infty_{j=1} |f_j|^{r}\bigg)^{\frac{1}{r}} \chi_{l}\bigg\|_{L^{q(\cdot)}(w)} \nonumber\\
& \quad \les \sum^{k-2}_{l=-\infty} 2^{(l-k)n\delta_{2}} \bigg\|\bigg(\sum^\infty_{j=1} |f_j|^{r}\bigg)^{\frac{1}{r}} \chi_{l}\bigg\|_{L^{q(\cdot)}(w)}.
\end{align}
If $l  \geq k+2$, we deduce that
$|x-y| \geq |y|-|x| > 2^{l-2},\,x\in D_k,\, y \in D_l.$
 For $\forall x \in D_k$, since the sublinear operator $T$ satisfies (\ref{vf-t2}),
then  we have
\[ |Tf^l_j(x)| \lesssim 2^{-ln} \int_{\rn} |f^l_j(y)| {\rm d}y.\]
Thus, by Minkowski's inequality, H\"{o}lder's inequality, Definition \ref{dg1} and Lemma \ref{tt-L5}, we obtain

\begin{align}\label{vsgh-2}
&\bigg\|\bigg(\sum^\infty_{j=1}\bigg| \sum^{\infty}_{l=k+2} Tf^l_j\bigg|^r\bigg)^{\frac{1}{r}} \chi_k \bigg\|_{L^{q(\cdot)}(w)} \no \\
& \quad \lesssim \bigg\|\bigg(\sum^\infty_{j=1}\bigg(\sum^{\infty}_{l=k+2} 2^{-ln} \int_{\rn} |f^l_j(y)|{\rm d}y\bigg)^{r}\bigg)^{\frac{1}{r}}  \chi_{k}\bigg\|_{L^{q(\cdot)}(w)} \nonumber\\
& \quad \lesssim \bigg\|\sum^{\infty}_{l=k+2} 2^{-ln} \int_{\rn} \bigg(\sum^\infty_{j=1}|f^l_j(y)|^{r}\bigg)^{\frac{1}{r}}{\rm d}y \chi_{k}\bigg\|_{L^{q(\cdot)}(w)} \no\\
& \quad\les \sum^{\infty}_{l=k+2} 2^{-ln} \|\chi_{B_k}\|_{L^{q(\cdot)}(w)} \bigg\|\bigg(\sum^\infty_{j=1} |f_j|^{r}\bigg)^{\frac{1}{r}} w \chi_{l}\bigg\|_{L^{q(\cdot)}(\rn)} \|\chi_{l}w^{-1}\|_{L^{q'(\cdot)}(\rn)} \nonumber\\
& \quad \les \sum^{\infty}_{l=k+2} 2^{-ln} |B_l| \|\chi_{B_k}\|_{L^{q(\cdot)}(w)}  \|\chi_{B_{l}}\|^{-1}_{L^{q(\cdot)}(w)} \bigg\|\bigg(\sum^\infty_{j=1} |f_j|^{r}\bigg)^{\frac{1}{r}} \chi_{l}\bigg\|_{L^{q(\cdot)}(w)} \nonumber\\
& \quad \les \sum^{\infty}_{l=k+2} 2^{(k-l)n\delta_1} \bigg\|\bigg(\sum^\infty_{j=1} |f_j|^{r}\bigg)^{\frac{1}{r}} \chi_{l}\bigg\|_{L^{q(\cdot)}(w)}.
\end{align}

Estimate $E_1$. By $\varepsilon_1= n\delta_2-\alpha(0)>0$ and (\ref{vsgh-1}) and H\"{o}lder's inequality, we obtain
\begin{align*}
&\bigg( \delta^\theta \sum_{k=-\infty}^{k_0} 2^{k\alpha(0)p(1+\delta)}\bigg( \sum^{k-2}_{l=-\infty} 2^{(l-k)n\delta_2} \bigg\|\bigg(\sum^\infty_{j=1} |f_j|^{r}\bigg)^{\frac{1}{r}} \chi_{l}\bigg\|_{L^{q(\cdot)}(w)} \bigg)^{p(1+\delta)} \bigg)^{\frac{1}{p(1+\delta)}} \\
& \les \bigg( \delta^\theta \sum_{k=-\infty}^{k_0}\bigg( \sum^{k-2}_{l=-\infty} 2^{\alpha(0)l} 2^{(l-k)\varepsilon_1} \bigg\|\bigg(\sum^\infty_{j=1} |f_j|^{r}\bigg)^{\frac{1}{r}} \chi_{l}\bigg\|_{L^{q(\cdot)}(w)} \bigg)^{p(1+\delta)} \bigg)^{\frac{1}{p(1+\delta)}} \\
& \les \bigg( \delta^\theta \sum_{k=-\infty}^{k_0}  \bigg( \sum^{k-2}_{l=-\infty} 2^{\alpha(0)lp(1+\delta)} 2^{(l-k)\varepsilon_1 p(1+\delta)/2}\bigg\|\bigg(\sum^\infty_{j=1} |f_j|^{r}\bigg)^{\frac{1}{r}} \chi_{l}\bigg\|_{L^{q(\cdot)}(w)}^{p(1+\delta)} \bigg)\\
&\quad\times  \bigg( \sum^{k-2}_{l=-\infty} 2^{(l-k)\varepsilon_1 (p(1+\delta))^{\prime}/2} \bigg)^{\frac{p(1+\delta)}{(p(1+\delta))^{\prime}}} \bigg)^{\frac{1}{p(1+\delta)}} \\
& \les \bigg( \delta^\theta \sum_{l=-\infty}^{k_0} 2^{\alpha(0)lp(1+\delta)}  \bigg\|\bigg(\sum^\infty_{j=1} |f_j|^{r}\bigg)^{\frac{1}{r}} \chi_{l}\bigg\|_{L^{q(\cdot)}(w)}^{p(1+\delta)} \bigg)^{\frac{1}{p(1+\delta)}}.
\end{align*}
Hence
\[E_1  \les   \bigg\| \bigg( \sum^\infty_{j=1} |f_j|^r \bigg)^{\frac{1}{r}}\bigg\|_{M\dot{K}_{q(\cdot),\lambda}^{\alpha (\cdot),p),\theta}(w)}.\]

Estimate $E_2$. For $k-1 \leq l \leq k+1$, $\forall x \in D_k$, since $T$ satisfies (\ref{vf-t1}),  then by  Minkowski's inequality, we obtain
\begin{align}\label{vsgh-3}
\bigg\|\bigg(\sum^\infty_{j=1}\bigg|\sum^{k+1}_{l=k-1} Tf^l_j\bigg|^r\bigg)^{\frac{1}{r}} \chi_k \bigg\|_{L^{q(\cdot)}(w)}
&\lesssim \bigg\|\sum^{k+1}_{l=k-1} \bigg( \sum^\infty_{j=1} |Tf^l_j|^r\bigg)^{\frac{1}{r}} \chi_k \bigg\|_{L^{q(\cdot)}(w)}  \nonumber\\
&\lesssim \sum^{k+1}_{l=k-1}\bigg\| \bigg( \sum^\infty_{j=1} |Tf^l_j|^r\bigg)^{\frac{1}{r}} \chi_k \bigg\|_{L^{q(\cdot)}(w)}  \nonumber\\
&\lesssim \sum^{k+1}_{l=k-1}\bigg\| \bigg( \sum^\infty_{j=1} |f_j|^r\bigg)^{\frac{1}{r}} \chi_l\bigg\|_{L^{q(\cdot)}(w)}.
\end{align}
Thus, by (\ref{vsgh-3}), we have
\begin{align*}
E_2 &\les \sup_{\delta>0} \sup_{k_0\leq 0,k_0\in \zz}2^{-k_0\lambda} \bigg( \delta^\theta \sum_{k=-\infty}^{k_0} 2^{k\alpha(0)p(1+\delta)} \sum^{k+1}_{l=k-1}\bigg\| \bigg( \sum^\infty_{j=1} |f_j|^r\bigg)^{\frac{1}{r}} \chi_l\bigg\|_{L^{q(\cdot)}(w)}^{p(1+\delta)} \bigg)^{\frac{1}{p(1+\delta)}} \\
&\les \sup_{\delta>0} \sup_{k_0\leq 0,k_0\in \zz}2^{-k_0\lambda} \bigg( \delta^\theta \sum_{l=-\infty}^{k_0+1} 2^{l\alpha(0)p(1+\delta)} \bigg\| \bigg( \sum^\infty_{j=1} |f_j|^r\bigg)^{\frac{1}{r}} \chi_l\bigg\|_{L^{q(\cdot)}(w)}^{p(1+\delta)} \bigg)^{\frac{1}{p(1+\delta)}} \\
& \les  \bigg\| \bigg( \sum^\infty_{j=1} |f_j|^r \bigg)^{\frac{1}{r}}\bigg\|_{M\dot{K}_{q(\cdot),\lambda}^{\alpha (\cdot),p),\theta}(w)}.
\end{align*}

Estimate $E_3$. By (\ref{vsgh-2}), we obtain
\[
E_3 \les \sup_{\delta>0} \sup_{k_0\leq 0,k_0\in \zz} 2^{-k_0\lambda} E_{3,1} + \sup_{\delta>0} \sup_{k_0\leq 0,k_0\in \zz} 2^{-k_0\lambda} E_{3,2},\]
where 
\[E_{3,1}:= \bigg( \delta^\theta \sum_{k=-\infty}^{k_0} 2^{k\alpha(0)p(1+\delta)} \bigg( \sum_{l=k+2}^{-1} 2^{(k-l)n\delta_1} \bigg\|\bigg(\sum^\infty_{j=1} |f_j|^{r}\bigg)^{\frac{1}{r}} \chi_{l}\bigg\|_{L^{q(\cdot)}(w)} \bigg)^{p(1+\delta)} \bigg)^{\frac{1}{p(1+\delta)}}\]
and \[E_{3,2}:=\bigg( \delta^\theta \sum_{k=-\infty}^{k_0} 2^{k\alpha(0)p(1+\delta)} \bigg( \sum_{l=0}^\infty 2^{(k-l)n\delta_1}\bigg\|\bigg(\sum^\infty_{j=1} |f_j|^{r}\bigg)^{\frac{1}{r}} \chi_{l}\bigg\|_{L^{q(\cdot)}(w)} \bigg)^{p(1+\delta)} \bigg)^{\frac{1}{p(1+\delta)}}.\]

Estimate $E_{3,1} $. By $\varepsilon_2=n\delta_1+\alpha(0)>0$ and H\"{o}lder's inequality, we have 
\begin{align*}
E_{3,1} 
& \les  \bigg( \delta^\theta \sum_{k=-\infty}^{k_0}  \bigg( \sum_{l=k+2}^{-1} 2^{\alpha(0)l} 2^{(k-l)\varepsilon_2}\bigg\|\bigg(\sum^\infty_{j=1} |f_j|^{r}\bigg)^{\frac{1}{r}} \chi_{l}\bigg\|_{L^{q(\cdot)}(w)} \bigg)^{p(1+\delta)} \bigg)^{\frac{1}{p(1+\delta)}} \\
& \les  \bigg( \delta^\theta \sum_{k=-\infty}^{k_0}  \bigg( \sum_{l=k+2}^{-1} 2^{\alpha(0)lp(1+\delta)} 2^{(k-l)\varepsilon_2 p(1+\delta)/2} \\
& \quad \times \bigg\|\bigg(\sum^\infty_{j=1} |f_j|^{r}\bigg)^{\frac{1}{r}} \chi_{l}\bigg\|_{L^{q(\cdot)}(w)}^{p(1+\delta)} \bigg)  \bigg( \sum_{l=k+2}^{-1} 2^{(k-l)\varepsilon_2 (p(1+\delta))^{\prime}/2} \bigg)^{\frac{p(1+\delta)}{(p(1+\delta))^{\prime}}} \bigg)^{\frac{1}{p(1+\delta)}} \\
& \les \bigg( \delta^\theta \sum_{l=-\infty}^{k_0} 2^{\alpha(0)lp(1+\delta)}  \bigg\|\bigg(\sum^\infty_{j=1} |f_j|^{r}\bigg)^{\frac{1}{r}} \chi_{l}\bigg\|_{L^{q(\cdot)}(w)}^{p(1+\delta)} \bigg)^{\frac{1}{p(1+\delta)}}.
\end{align*}

Estimate $E_{3,2} $. By  $\varepsilon_3=n\delta_1+\alpha_\infty>0$ and  H\"{o}lder's inequality, we have 
\begin{align*}
E_{3,2} 
& \les \bigg( \delta^\theta \sum_{k=-\infty}^{k_0} 2^{k(\alpha(0)+n\delta_1)p(1+\delta)}  \bigg( \sum_{l=0}^\infty 2^{-ln\delta_1} \bigg\|\bigg(\sum^\infty_{j=1} |f_j|^{r}\bigg)^{\frac{1}{r}} \chi_{l}\bigg\|_{L^{q(\cdot)}(w)} \bigg)^{p(1+\delta)} \bigg)^{\frac{1}{p(1+\delta)}} \\
&   \les \bigg( \delta^\theta \bigg( \sum_{l=0}^\infty 2^{l\alpha_\infty} 2^{-l\varepsilon_3} \bigg\|\bigg(\sum^\infty_{j=1} |f_j|^{r}\bigg)^{\frac{1}{r}} \chi_{l}\bigg\|_{L^{q(\cdot)}(w)} \bigg)^{p(1+\delta)} \bigg)^{\frac{1}{p(1+\delta)}}\\
& \les \bigg( \delta^\theta \bigg( \sum_{l=0}^\infty 2^{l\alpha_\infty p(1+\delta)}  \bigg\|\bigg(\sum^\infty_{j=1} |f_j|^{r}\bigg)^{\frac{1}{r}} \chi_{l}\bigg\|_{L^{q(\cdot)}(w)}^{p(1+\delta)} \bigg) \bigg( \sum_{l=0}^\infty 2^{-l\varepsilon_3 (p(1+\delta))^{\prime}} \bigg)^{\frac{p(1+\delta)}{(p(1+\delta))^{\prime}}} \bigg)^{\frac{1}{p(1+\delta)}}.
\end{align*}
Thus, we have 
\[E_3\les  \bigg\| \bigg( \sum^\infty_{j=1} |f_j|^r \bigg)^{\frac{1}{r}}\bigg\|_{M\dot{K}_{q(\cdot),\lambda}^{\alpha (\cdot),p),\theta}(w)}.\]

Estimate $G_1$. By (\ref{vsgh-1}), we obtain
\[G_1 \les \sup_{\delta>0} \sup_{k_0 > 0,k_0\in \zz} G_{1,1}  + \sup_{\delta>0} \sup_{k_0 > 0,k_0\in \zz} G_{1,2},\]
where
\[G_{1,1} :=2^{-k_0\lambda} \bigg( \delta^\theta \sum_{k=0}^{k_0} 2^{k\alpha_\infty p(1+\delta)} \bigg( \sum^{-1}_{l=-\infty} 2^{(l-k)n\delta_2} \bigg\|\bigg(\sum^\infty_{j=1} |f_j|^{r}\bigg)^{\frac{1}{r}} \chi_{l}\bigg\|_{L^{q(\cdot)}(w)} \bigg)^{p(1+\delta)} \bigg)^{\frac{1}{p(1+\delta)}} \]
and \[G_{1,2}:= 2^{-k_0\lambda} \bigg( \delta^\theta \sum_{k=0}^{k_0} 2^{k\alpha_\infty p(1+\delta)} \bigg( \sum^{k-2}_{l=0} 2^{(l-k)n\delta_2} \bigg\|\bigg(\sum^\infty_{j=1} |f_j|^{r}\bigg)^{\frac{1}{r}} \chi_{l}\bigg\|_{L^{q(\cdot)}(w)} \bigg)^{p(1+\delta)} \bigg)^{\frac{1}{p(1+\delta)}}\]

Estimate $G_{1,1}$. By $\varepsilon_4=n\delta_2-\alpha(0)>0$ and  H\"{o}lder's inequality, we have 
\begin{align*}
G_{1,1} & \les  \bigg( \delta^\theta \sum_{k=0}^{k_0} 2^{k(\alpha_\infty-n\delta_2) p(1+\delta)} \bigg( \sum^{-1}_{l=-\infty} 2^{ln\delta_2} \bigg\|\bigg(\sum^\infty_{j=1} |f_j|^{r}\bigg)^{\frac{1}{r}} \chi_{l}\bigg\|_{L^{q(\cdot)}(w)} \bigg)^{p(1+\delta)} \bigg)^{\frac{1}{p(1+\delta)}} \\
&\les 2^{-k_0\lambda} \bigg( \delta^\theta \bigg( \sum^{-1}_{l=-\infty} 2^{ln\delta_2} \bigg\|\bigg(\sum^\infty_{j=1} |f_j|^{r}\bigg)^{\frac{1}{r}} \chi_{l}\bigg\|_{L^{q(\cdot)}(w)} \bigg)^{p(1+\delta)} \bigg)^{\frac{1}{p(1+\delta)}} \\
& \les 2^{-k_0\lambda} \bigg( \delta^\theta \bigg( \sum^{-1}_{l=-\infty} 2^{l\alpha(0)} \bigg\|\bigg(\sum^\infty_{j=1} |f_j|^{r}\bigg)^{\frac{1}{r}} \chi_{l}\bigg\|_{L^{q(\cdot)}(w)} 2^{l\varepsilon_4} \bigg)^{p(1+\delta)}  \bigg)^{\frac{1}{p(1+\delta)}} \\
& \les 2^{-k_0\lambda} \bigg( \delta^\theta \bigg( \sum^{-1}_{l=-\infty} 2^{l\alpha(0)p(1+\delta)} \bigg\|\bigg(\sum^\infty_{j=1} |f_j|^{r}\bigg)^{\frac{1}{r}} \chi_{l}\bigg\|_{L^{q(\cdot)}(w)}^{p(1+\delta)} \bigg)  \bigg( \sum^{-1}_{l=-\infty} 2^{l\varepsilon_4(p(1+\delta))^{\prime}} \bigg)^{\frac{p(1+\delta)}{(p(1+\delta))^{\prime}}}  \bigg)^{\frac{1}{p(1+\delta)}} \\
& \les 2^{-k_0\lambda} \bigg( \delta^\theta \bigg( \sum^{-1}_{l=-\infty} 2^{l\alpha(0)p(1+\delta)} \bigg\|\bigg(\sum^\infty_{j=1} |f_j|^{r}\bigg)^{\frac{1}{r}} \chi_{l}\bigg\|_{L^{q(\cdot)}(w)}^{p(1+\delta)} \bigg) \bigg)^{\frac{1}{p(1+\delta)}}.
\end{align*}

Estimate $G_{1,2}$. By using the same argument as $E_1$ and the fact that $\varepsilon_5= n\delta_2-\alpha_\infty>0$,  we get the desired result. Thus, we obtain
\[G_1  \les  \bigg\| \bigg( \sum^\infty_{j=1} |f_j|^r \bigg)^{\frac{1}{r}}\bigg\|_{M\dot{K}_{q(\cdot),\lambda}^{\alpha (\cdot),p),\theta}(w)}.\]

Estimate $G_2$. By (\ref{vsgh-3}), we obtain
\begin{align*}
G_2 &\les \sup_{\delta>0} \sup_{k_0 > 0,k_0\in \zz}2^{-k_0\lambda}  \bigg( \delta^\theta \sum_{k=0}^{k_0} 2^{k\alpha_\infty p(1+\delta)} \sum^{k+1}_{l=k-1}\bigg\| \bigg( \sum^\infty_{j=1} |f_j|^r\bigg)^{\frac{1}{r}} \chi_l\bigg\|_{L^{q(\cdot)}(w)}^{p(1+\delta)} \bigg)^{\frac{1}{p(1+\delta)}} \\
&\les \sup_{\delta>0} \sup_{k_0\leq 0,k_0\in \zz}2^{-k_0\lambda} \bigg( \delta^\theta \sum_{l=0}^{k_0+1} 2^{l\alpha_\infty p(1+\delta)} \bigg\| \bigg( \sum^\infty_{j=1} |f_j|^r\bigg)^{\frac{1}{r}} \chi_l\bigg\|_{L^{q(\cdot)}(w)}^{p(1+\delta)} \bigg)^{\frac{1}{p(1+\delta)}} \\
& \les  \bigg\| \bigg( \sum^\infty_{j=1} |f_j|^r \bigg)^{\frac{1}{r}}\bigg\|_{M\dot{K}_{q(\cdot),\lambda}^{\alpha (\cdot),p),\theta}(w)}.
\end{align*}

Estimate $G_3$. By using the same argument as $E_{3,1}$, and using the fact that $\varepsilon_6=n\delta_1+\alpha_\infty>0$, we get the desired result.

Combining the estimates of $E$, $F$ and $G$, we get (\ref{vf-t3}).
\end{proof}

\section{Characterizations of  Herz-Morrey-Lizorkin-Triebel Spaces}\label{wbl-s-3}
In this section,  we  introduce weighted grand Herz-Morrey-Triebel-Lizorkin spaces with variable exponent, and obtain their equivalent quasi-norms.
Below, for convenience, we denote by $M\dot{K}_{q(\cdot),\lambda,w}^{\alpha (\cdot),p),\theta}$ the   weighted grand Herz-Morrey spaces $M\dot{K}_{q(\cdot),\lambda}^{\alpha (\cdot),p),\theta}(w)$.
To go on, we recall some notations.

Let $\mathcal{S}(\rn)$ be the Schwartz space of rapidly decreasing functions on $\rn$ and $\mathcal{S}^{\prime}(\rn)$ the space of temperate distributions on $\rn$.
For $f \in \mathcal{S}(\rn)$, let $\mathcal{F} f$ or $\hat{f}$ denote the Fourier transform of $f$ defined by
\[ \mathcal{F} f(\xi) =\widehat{f}(\xi) :=(2 \pi)^{-n / 2} \int_{\rn} e^{-i x \xi} f(x) {\rm d}x, \ \xi \in \rn\]
while $f^\vee(\xi)= \widehat{f} (-\xi)$ denote the inverse Fourier transform of $f$.

\begin{defn}\label{bt-D1} 
Let $\phi_0$ be a function in $\cs(\rn)$ satisfying
$\widehat{\phi} _0(x)=1$ for $|x| \leq 1$ and $\widehat{\phi} _0(x)=0$ for $|x| \geq 2$.
We put $\widehat{\phi} (x)=\widehat{\phi} _0-\widehat{\phi} _0(2x)$ and $\widehat{\phi} _j(x)=\widehat{\phi} (2^{-j}x)$  for all $j \in \nn$.
Then $\{\widehat{\phi} _j\}_{j \in \nn_0}$ is a smooth dyadic resolution of unity and
$\sum_{j=0}^\infty \widehat{\phi} _j(x) =1 \quad \text{for all $x \in \rn$}.$
\end{defn}

\begin{defn}\label{bt-D2}
 Let $q(\cdot)\in\cp(\rn),$ $p\in(1,\infty)$, $\lambda \in [0, \infty)$, $\alpha(\cdot) \in L^\infty(\rn)$, $\theta>0$, $s \in \rr$ and $\beta \in (0,\infty)$.
Furthermore let $\{\phi_j\}^\infty_{j=0}$ be the system as Definition \ref{bt-D1} and $w$ be a weight.
 The  weighted grand Herz-Morrey-Triebel-Lizorkin space with variable exponents $M\dot{K}_{q(\cdot),\lambda,w}^{\alpha (\cdot),p),\theta}F^s_\beta(\rn)$  is defined to be the set of all $f \in \cs^{\prime}(\rn)$ such that
\[\|f\|_{M\dot{K}_{q(\cdot),\lambda,w}^{\alpha (\cdot),p),\theta}F^s_\beta}:=\|\{ 2^{js}\phi_j \ast f\}^\infty_{j=0}\|_{M\dot{K}_{q(\cdot),\lambda,w}^{\alpha (\cdot),p),\theta}(\ell^\beta)}<\infty.\]

\end{defn}

\begin{defn}
A pair of functions $(\phi,\Phi)$ is said to be admissible, if $\phi,$ $\Phi \in \mathcal{S}(\rn)$ satisfy
\begin{equation}\label{vsgh-15}
\supp \widehat{\phi} \subseteq \{ \xi \in \rn : 2^{-1} \leq |\xi| \leq 2\} \text{ and } | \widehat{\phi} | \geq C>0 \text{ when } \frac{3}{5} \leq |\xi| \leq \frac{5}{3}
\end{equation}
and
\begin{equation}\label{vsgh-16}
\supp \widehat{\Phi} \subseteq \{ \xi \in \rn :  |\xi| \leq 2\} \text{ and } | \widehat{\Phi} | \geq C>0 \text{ when } |\xi| \leq \frac{5}{3},
\end{equation}
where $C$ is a positive constant independent of $\xi \in \rn$. 
Throughout the article,  we set   $\tilde{\phi}(x):=\overline{{\phi(-x)}}$.
\end{defn}

The following two lemmas come from {\cite[Lemmas 2.41 and 2.45]{dd1}}.
\begin{lem}\label{tt-L3}
Let $x \in \rn$, $N > 0$, $m > n$ and $\omega \in \cs(\rn)$. Then there exists a positive
constant $C$ independent of $N$ and $x$ such that for all $f \in L^1_{\rm loc} (\rn)$,
\[ |\omega_N \ast f(x)| \leq C \mathcal{M}f(x), \]
where $\omega_N=N^n \omega(N\cdot)$.
\end{lem}

\begin{lem}\label{tt-L4}
Let $r, R, N > 0$, $m > n$ and $\vartheta, \omega \in \cs(\rn)$ with $\supp \cf \omega \subset \{\xi \in \rn : |\xi| \leq 2\}$. Then there exists $C = C(r, m, n) > 0$ such that for all $g \in \cs^{\prime}(\rn)$,
\[ |\vartheta_R \ast \omega_N \ast g(x)| \leq C \max\{1, (N/R)^m\} (\eta_{N,m} \ast |\omega_N \ast g|^r(x))^{1/r}, \quad x \in \rn, \]
where $\vartheta_R=R^n \vartheta(R\cdot)$, $\omega_N=N^n\omega(N \cdot)$ and $\eta_{N,m}=N^n(1+N|\cdot|)^{-m}$.
\end{lem}

Following \cite{fj}, given an admissible pair $(\phi,\Phi)$ we can select another admissible pair $(\psi,\Psi)$ such that\begin{equation}\label{vsgh-17}
\cf\tilde{\Phi}(\xi) \cf\Psi (\xi)+\sum_{\nu =1}^\infty \cf\tilde{\phi}(2^{-\nu}\xi) \cf \psi(2^{-\nu}\xi)=1, \quad \xi \in \rn.
\end{equation}

The following lemma is the Calder\'on reproducing formula; see {\cite[(12.4)]{fj}} and {\cite[Lemma 2.3]{ysy-1}}.
\begin{lem}\label{tt-L6}
Let $\Phi, \Psi \in \cs(\rn)$ satisfy (\ref{vsgh-15}) and $\phi, \psi \in \cs(\rn)$ satisfy (\ref{vsgh-16}) such that (\ref{vsgh-17}) holds. Then for all $f \in \cs^{\prime}(\rn)$,
\begin{align*}
f &= \tilde{\Phi} \ast \Psi \ast f + \sum_{k=1}^\infty \tilde{\phi}_k \ast \psi_k \ast f \\
&= \sum_{m \in \zn} \tilde{\Phi}\ast f(m) \Psi_m + \sum_{k=1}^\infty 2^{-kn/2} \sum_{m\in \zn} \tilde{\phi}_k \ast f(2^{-k}m) \psi_{k,m}
\end{align*}
in $\cs^{\prime}(\rn)$, where $\Psi_m=\Psi(\cdot-m)$ and $\psi_{k,m}=2^{kn/2} \psi(2^k \cdot-m)$, $m\in \zn$, $k\in \nn$.
\end{lem}

\begin{thm}\label{vT2}
 Let $q(\cdot)\in\cp(\rn),$ $p\in(1,\infty)$, $\lambda \in [0, \infty)$, $\alpha(\cdot) \in L^{\infty}(\mathbb{R}^{n})\cap C^{\log}_0(\rn)\cap C^{\log}_{\infty}(\rn)$, $\theta \in (0,\infty)$, $\beta \in (0,\infty]$,  $s \in \rn$, $w \in A_{q(\cdot)}$, $- n\delta_1 <\alpha(0),\ \alpha_\infty <n\delta_2$,  where $\delta_1,\ \delta_2 \in (0,1)$ are the constants in Lemma \ref{tt-L5} for $q(\cdot)$.
 A tempered distribution $f$ belongs to $M\dot{K}_{q(\cdot),\lambda,w}^{\alpha (\cdot),p),\theta}F^s_\beta(\rn)$ if and only if $\|f\|^\star_{M\dot{K}_{q(\cdot),\lambda,w}^{\alpha (\cdot),p),\theta}F^s_\beta(\rn)}<\infty $, where
 \[\|f\|^\star_{M\dot{K}_{q(\cdot),\lambda,w}^{\alpha (\cdot),p),\theta}F^s_\beta(\rn)}:=\|\{ 2^{js}\phi_j \ast f\}^\infty_{j=0}\|_{M\dot{K}_{q(\cdot),\lambda,w}^{\alpha (\cdot),p),\theta}(\ell^\beta)},\] 
   $\phi$ and $\Phi$  satisfy  (\ref{vsgh-15}) and (\ref{vsgh-16}), respectively and  $\phi_j=2^{jn} \phi(2^j \cdot)$, $j\in \nn$, when $j=0$, $\phi_0$ is replaced by $\Phi$.
 Furthermore, the quasi-norms $\|f\|_{M\dot{K}_{q(\cdot),\lambda,w}^{\alpha (\cdot),p),\theta}F^s_\beta(\rn)}$ and $\|f\|^\star_{M\dot{K}_{q(\cdot),\lambda,w}^{\alpha (\cdot),p),\theta}F^s_\beta(\rn)}$ are equivalent.
\end{thm}
\begin{proof}
Let $\Phi, \Psi \in \cs(\rn)$ satisfy (\ref{vsgh-15}) and $\phi, \psi \in \cs(\rn)$ satisfy (\ref{vsgh-16}) such that (\ref{vsgh-17}) holds. 
By Lemma \ref{tt-L6} and inspecting the support conditions, we have
\begin{equation*}
 \phi_k\ast f= \sum_{j=k-1}^{k+1}  \phi_k \ast \tilde{\phi}_j \ast \psi_j \ast f +
\begin{cases}
0 & \text{if } k \geq 3  \\
 \phi_k \ast \tilde{\Phi} \ast \Psi \ast f,& \text{if } k \in\{1,2\}
\end{cases}
\end{equation*}
and 
\[  \phi_0\ast f=    \phi_0 \ast \tilde{\phi}_1 \ast \psi_1 \ast f +  \phi_0 \ast \tilde{\Phi} \ast \Psi \ast f .\]
For $j \in \{k-1,k,k+1\}$, $k \geq 3$, by Lemmas \ref{tt-L4} and \ref{tt-L3}, we obtain
\[  | \phi_k \ast \tilde{\phi}_j \ast \psi_j \ast f| \les \mathcal{M}_t(\tilde{\phi}_j \ast f), \quad t\in(0,\infty).\]
Similarly, when $k \in \{0,1,2\}$, we have
\[| \phi_k \ast \tilde{\Phi} \ast \Psi \ast f| + | \phi_0 \ast \tilde{\phi}_1 \ast \psi_1 \ast f| \les \mathcal{M}_t(\tilde{\Phi}\ast f) + \mathcal{M}_t(\tilde{\phi}_1 \ast f), \quad t\in(0,\infty).\]
We take $t \in (0,\min\{q^-,\beta\})$, then by Corollary \ref{vC1}, we obtain
\[\|f\|^\star_{M\dot{K}_{q(\cdot),\lambda,w}^{\alpha (\cdot),p),\theta}F^s_\beta(\rn)} \les \|f\|_{M\dot{K}_{q(\cdot),\lambda,w}^{\alpha (\cdot),p),\theta}F^s_\beta(\rn)}.\]
The opposite inequality is followed by the same arguments. Notice that we use the smooth resolution of unity (Definition \ref{bt-D1}). This completes the proof.
\end{proof}

\begin{cor}
Let $\{\varpi_j\}_{j \in \nn_0}$ and $\{\phi_j\}_{j \in \nn_0}$ be two resolutions of unity.
 Let $q(\cdot)\in\cp(\rn),$ $p\in(1,\infty)$, $\lambda \in [0, \infty)$, $\alpha(\cdot) \in L^{\infty}(\mathbb{R}^{n})\cap C^{\log}_0(\rn)\cap C^{\log}_{\infty}(\rn)$, $\theta \in (0,\infty)$, $\beta \in (0,\infty]$, $s\in \rn$, $w \in A_{q(\cdot)}$, $- n\delta_1 <\alpha(0),\ \alpha_\infty <n\delta_2$,  where $\delta_1,\ \delta_2 \in (0,1)$ are the constants in Lemma \ref{tt-L5} for $q(\cdot)$. Then
 \[\|f\|^\varpi_{M\dot{K}_{q(\cdot),\lambda,w}^{\alpha (\cdot),p),\theta} F^s_\beta (\rn)} \sim \|f\|^\phi_{M\dot{K}_{q(\cdot),\lambda,w}^{\alpha (\cdot),p),\theta}F^s_\beta (\rn)} \sim \|f\|^\star_{M\dot{K}_{q(\cdot),\lambda,w}^{\alpha (\cdot),p),\theta}F^s_\beta (\rn)}.\]
\end{cor}

Let $\phi_j\in \mathcal{S}(\rn),$ $j\in \nn_0,$ and $a>0.$ For each $f \in \mathcal{S}^{\prime}(\rn), $ the Peetre maximal functions for $f$ are defined by
\begin{equation}\label{vsgh-19}
 \phi^{\ast,a}_j f(x) := \sup_{y \in \rn} \frac{ | \phi_j \ast f(y)| }{(1+2^j |x-y|)^a}, \quad j \in \nn_0.
\end{equation}

\begin{thm}
Let $\{\phi_j\}_{j \in \nn_0}$ be a resolutions of unity.
 Let $q(\cdot)\in\cp(\rn),$ $p\in(1,\infty)$, $\lambda \in [0, \infty)$, $\alpha(\cdot) \in L^{\infty}(\mathbb{R}^{n})\cap C^{\log}_0(\rn)\cap C^{\log}_{\infty}(\rn)$, $\theta \in (0,\infty)$, $\beta \in (0,\infty]$, $s\in \rr$, $t\in (0,\min\{q^-,\beta\})$,  $w \in A_{q(\cdot)}$, $- n\delta_1 <\alpha(0),\ \alpha_\infty <n\delta_2$,  where $\delta_1,\ \delta_2 \in (0,1)$ are the constants in Lemma \ref{tt-L5} for $q(\cdot)$.  
  If $a t>n$. Then
\[\|f\|^\ast_{M\dot{K}_{q(\cdot),\lambda,w}^{\alpha (\cdot),p),\theta}F^s_\beta(\rn)}:=  \bigg\| \bigg( \sum_{j=0}^\infty   2^{js\beta} |\phi^{\ast,a}_jf|^\beta \bigg)^{1/\beta} \bigg\|_{M\dot{K}_{q(\cdot),\lambda}^{\alpha (\cdot),p),\theta}(w)}\]
is equivalent (quasi-)norms in $M\dot{K}_{q(\cdot),\lambda,w}^{\alpha (\cdot),p),\theta}F^s_\beta(\rn)$ .

\end{thm}
\begin{proof}
For any $f \in \cs^{\prime}(\rn)$, we show that 
\[\|f\|^\ast_{M\dot{K}_{q(\cdot),\lambda,w}^{\alpha (\cdot),p),\theta}F^s_\beta(\rn)} \les \|f\|^\star_{M\dot{K}_{q(\cdot),\lambda,w}^{\alpha (\cdot),p),\theta}F^s_\beta(\rn)} \les \|f\|^\ast_{M\dot{K}_{q(\cdot),\lambda,w}^{\alpha (\cdot),p),\theta}F^s_\beta(\rn)} .\]
By the definition of Peetre’s maximal operator (\ref{vsgh-19}), we find that
\[ |  \phi_j \ast f (x)| \les \phi^{\ast,a}_j f(x)\]
which implies that, for any $f \in \cs^{\prime}(\rn)$,
\[\|f\|^\star_{M\dot{K}_{q(\cdot),\lambda,w}^{\alpha (\cdot),p),\theta}F^s_\beta(\rn)} \les \|f\|^\ast_{M\dot{K}_{q(\cdot),\lambda,w}^{\alpha (\cdot),p),\theta}F^s_\beta(\rn)}.\]
 It remains to show that, for any $f \in \cs^{\prime}(\rn)$,
\[\|f\|^\ast_{M\dot{K}_{q(\cdot),\lambda,w}^{\alpha (\cdot),p),\theta}F^s_\beta(\rn)} \les \|f\|^\star_{M\dot{K}_{q(\cdot),\lambda,w}^{\alpha (\cdot),p),\theta}F^s_\beta(\rn)} .\]
We take  $t\in(0,\min\{q^-,\beta\})$, $a>n/t$ and $k \in \nn_0$.
By Lemmas \ref{tt-L4}, we obtain
\begin{equation}\label{vsgh-18}
 \phi_j \ast f(y) \les (\eta_{j,a t} \ast | \phi_j \ast f|^t(y))^{1/t}, \quad j\in \nn_0,\ y \in \rn.
\end{equation}
Divide both sides of (\ref{vsgh-18}) by $(1+2^j|x-y|)^a$, in the right-hand side we use the inequality,
\[ (1+2^j |x-z|)^a \leq (1+2^j |x-y|)^{a} (1+2^j |y-z|)^a , \quad \text{for all }x,y,z \in \rn,\]
in the left-hand side take the supremum over $y \in \rn$,
\begin{align*}
(\phi^{\ast,a}_j f(x))^t 
&\les \sup_{y \in \rn} \int_{\rn} \frac{| \phi_j \ast f(z)|^t}{(1+2^j |y-z|)^{at}} {\rm d}z \frac{1}{(1+2^j |x-y|)^{at}} \\
&\les  \int_{\rn} \frac{| \phi_j \ast f(z)|^t}{(1+2^j |x-z|)^{at}} {\rm d}z  \\
& \sim \eta_{j,at} \ast | \phi_j \ast f|^t(x)\\
& \les (\mathcal{M}_t ( \phi_j \ast f))^t,
\end{align*}
where we use Lemma \ref{tt-L3}.
By Corollary \ref{vC1} and Theorem \ref{vT2}, we obtain
\[\|f\|^\ast_{M\dot{K}_{q(\cdot),\lambda,w}^{\alpha (\cdot),p),\theta}F^s_\beta(\rn)}  \les \bigg\| \bigg(\sum_{j=0}^\infty 2^{js\beta} | \phi_j \ast f|^\beta \bigg)^{1/\beta} \bigg\|_{M\dot{K}_{q(\cdot),\lambda,w}^{\alpha (\cdot),p),\theta}} \les \|f\|^\star_{M\dot{K}_{q(\cdot),\lambda,w}^{\alpha (\cdot),p),\theta}F^s_\beta(\rn)} .\]
This completes the proof.
\end{proof}

\begin{lem}\label{tt-L1}
 Let $q(\cdot)\in\cp(\rn),$ $p\in(1,\infty)$, $\lambda \in [0, \infty)$, $\alpha(\cdot) \in L^\infty(\rn)$, $\theta \in (0,\infty)$, $\beta  \in (0,\infty]$ and $w\in A_{q(\cdot)}$.
For any sequence $\{g_j\}^\infty_{j=0}$ of
nonnegative measurable functions on $\rn$ denote
\[G_j(x)=\sum^\infty_{k=0} 2^{-|k-j|\delta} g_k(x),\quad x\in \rn.\]
Then there is a positive constant $C=C(q(\cdot),q(\cdot),\delta)$ such that
\begin{equation}\label{vsgh-7}
\big\|\|\{G_j\}\|_{\ell^\beta_j}\big\|_{M\dot{K}_{q(\cdot),\lambda,w}^{\alpha (\cdot),p),\theta}} \leq
C\big\|\|\{g_k\}\|_{\ell^\beta_k}\big\|_{M\dot{K}_{q(\cdot),\lambda,w}^{\alpha (\cdot),p),\theta}}.
\end{equation}
\end{lem}
\begin{proof} By Lemma 2 in \cite{r1}, taking the $M\dot{K}_{q(\cdot),\lambda,w}^{\alpha (\cdot),p),\theta}$-norms on both sides of the inequality, we get (\ref{vsgh-7}).
\end{proof}

Let $k_0$, $k \in \cs(\rn)$ and  $S  \geq -1$ integer an  such that for an $\epsilon > 0$
\begin{equation}\label{vsgh-4}
\cf k(\xi)>0 \quad \text{for }  |\xi|<2\epsilon,
\end{equation}
\begin{equation}\label{vsgh-5}
\cf k(\xi)>0 \quad \text{for }  \epsilon/2<|\xi|<2\epsilon
\end{equation}
and
\begin{equation}\label{vsgh-6}
 \int_{\rn} x^\alpha k(x) {\rm d}x=0 \quad \text{for any } |\alpha| \leq S,
\end{equation}
where (\ref{vsgh-4}) and (\ref{vsgh-5}) are Tauberian conditions, while (\ref{vsgh-6}) are  moment conditions on $k$.
We recall the notation.
\[ k_t(x):= t^{-n} k(t^{-1}x), \quad k_j(x):=k_{2^{-j}}(x), \quad \text{for $t>0$ and $j \in \nn$}. \]
For any $a>0$, $f \in \cs^{\prime}(\rn)$ and $x \in \rn$, we define
\begin{equation}\label{vsgh-20}
k^{\ast,a}_jf(x) := \sup_{y \in \rn} \frac{|k_j \ast f(x+y)|}{(1+|y|/j)^a}, \quad j>0.
\end{equation}

\begin{thm}
 Let $q(\cdot)\in\cp(\rn),$ $p\in(1,\infty)$, $\lambda \in [0, \infty)$, $\alpha(\cdot) \in L^{\infty}(\mathbb{R}^{n})\cap C^{\log}_0(\rn)\cap C^{\log}_{\infty}(\rn)$, $\theta \in (0,\infty)$, $\beta \in (0,\infty]$, $a\in \rr$, $r\in (0,\min\{q^-,\beta\})$, $s<S+1$, $w \in A_{q(\cdot)}$, $- n\delta_1 <\alpha(0),\ \alpha_\infty <n\delta_2$,  where $\delta_1,\ \delta_2 \in (0,1)$ are the constants in Lemma \ref{tt-L5} for $q(\cdot)$. Assume that $k_0,$ $k \in \cs(\rn)$ are functions satisfying (\ref{vsgh-4}), (\ref{vsgh-5}) and (\ref{vsgh-6}). 
 If $a r>n$, the space $M\dot{K}_{q(\cdot),\lambda,w}^{\alpha (\cdot),p),\theta}F^s_\beta(\rn)$ in Definition \ref{bt-D2} can be characterized by
 \[M\dot{K}_{q(\cdot),\lambda,w}^{\alpha (\cdot),p),\theta}F^s_\beta(\rn):= \big\{ f\in \cs^{\prime}(\rn): \|f\|^i_{M\dot{K}_{q(\cdot),\lambda,w}^{\alpha (\cdot),p),\theta}F^s_\beta(\rn)} <\infty \big\}, \quad i=1,\cdots,5,\]
where
\[ \|f\|^1_{M\dot{K}_{q(\cdot),\lambda,w}^{\alpha (\cdot),p),\theta}F^s_\beta(\rn)}  := \|k_0\ast f\|_{M\dot{K}_{q(\cdot),\lambda,w}^{\alpha (\cdot),p),\theta}}  + \bigg\| \bigg( \int^1_0  t^{-s\beta} |k_t \ast f|^\beta \frac{{\rm d}t}{t} \bigg)^{1/\beta} \bigg\|_{M\dot{K}_{q(\cdot),\lambda,w}^{\alpha (\cdot),p),\theta}},\]
\[\|f\|^2_{M\dot{K}_{q(\cdot),\lambda,w}^{\alpha (\cdot),p),\theta}F^s_\beta(\rn)} := \|k_0^\ast f\|_{M\dot{K}_{q(\cdot),\lambda,w}^{\alpha (\cdot),p),\theta}}  + \bigg\| \bigg( \int^1_0  t^{-s\beta} | k^{\ast,a}_tf|^\beta \frac{{\rm d}t}{t} \bigg)^{1/\beta} \bigg\|_{M\dot{K}_{q(\cdot),\lambda,w}^{\alpha (\cdot),p),\theta}}, \]

\begin{align*}
\|f\|^3_{M\dot{K}_{q(\cdot),\lambda,w}^{\alpha (\cdot),p),\theta}F^s_\beta(\rn)} & := \|k_0\ast f\|_{M\dot{K}_{q(\cdot),\lambda,w}^{\alpha (\cdot),p),\theta}}   \\
& \quad + \bigg\| \bigg( \int^1_0  t^{-s\beta} \int_{|z|<t} |k_t \ast f|(\cdot+z)^\beta {\rm d}z \frac{{\rm d}t}{t^{n+1}} \bigg)^{1/\beta} \bigg\|_{M\dot{K}_{q(\cdot),\lambda,w}^{\alpha (\cdot),p),\theta}},
\end{align*}
\[\|f\|^4_{M\dot{K}_{q(\cdot),\lambda,w}^{\alpha (\cdot),p),\theta}F^s_\beta(\rn)} :=  \bigg\| \bigg( \sum_{j=0}^\infty  2^{js\beta} |k^{\ast,a}_jf|^\beta \bigg)^{1/\beta} \bigg\|_{M\dot{K}_{q(\cdot),\lambda,w}^{\alpha (\cdot),p),\theta}}, \]
\[\|f\|^5_{M\dot{K}_{q(\cdot),\lambda,w}^{\alpha (\cdot),p),\theta}F^s_\beta(\rn)} :=  \bigg\| \bigg( \sum_{j=0}^\infty  2^{js\beta} |k_j \ast  f|^\beta \bigg)^{1/\beta} \bigg\|_{M\dot{K}_{q(\cdot),\lambda,w}^{\alpha (\cdot),p),\theta}}.\]
Then $\|f\|^i_{M\dot{K}_{q(\cdot),\lambda,w}^{\alpha (\cdot),p),\theta}F^s_\beta(\rn)},$ $i=1,\cdots,5,$  are equivalent (quasi-)norms in $M\dot{K}_{q(\cdot),\lambda,w}^{\alpha (\cdot),p),\theta}F^s_\beta(\rn)$.
\end{thm}

\begin{proof}
By an argument similar to that used in the proof of  {\cite[Theorem 2.6]{ut1}}, with {\cite[Theorem 2.1 and Lemma 2.13]{ut1}} replaced by Corollary \ref{vC1} and Lemma \ref{tt-L1},
we get the desired result. Due to the journal page limits, we leave the detail to the reader.
\end{proof}

\begin{rem} Similarly, one can introduce weighted grand Herz-Morrey-Besov spaces with variable exponents. We leave it to the interesting readers.
\end{rem}

\section*{FUNDING}

The work is supported by the National Natural Science Foundation of China (Grant Nos. 12161022 and 12061030) and the Science and Technology Project of Guangxi (Guike AD23023002).

\section*{CONFLICT OF INTEREST}

 The authors of this work declare that they have no conflicts of interest.

\end{document}